\documentclass[11pt]{amsart}%{birkjour}
 \usepackage[top=2.4cm, bottom=2.4cm, left=2.4cm, right=2.4cm]{geometry}
 \newtheorem{thm}{Theorem}[section]
 \newtheorem{cor}[thm]{Corollary}
 \newtheorem{lem}[thm]{Lemma}
 \newtheorem{prop}[thm]{Proposition}
 \theoremstyle{definition}
 \newtheorem{defn}[thm]{Definition}
 \theoremstyle{remark}
 \newtheorem{rem}[thm]{Remark}
 
 \numberwithin{equation}{section}

\newcommand{\norm}[1]{\left\Vert#1\right\Vert} \newcommand{\scal}[1]{\left<#1\right>}
  \newcommand{\R}{\mathbb{R}}  \newcommand{\C}{\mathbb{C}}
\newcommand{\Hq}{\mathbb{H}}  \newcommand{\D}{\mathbb{D}} \newcommand{\BC}{\mathbb{T}} 
  \newcommand{\FBC}{\mathcal{F}^{2,\nu }(\BC)} 
  \newcommand{\BargBC}{\mathcal{B}^{\sigma,\nu}_{\BC}}

 \newcommand{\eu}{{e_+}}  \newcommand{\es}{{e_-}}
\newcommand{\Hbc}{\mathcal{H}^{2,\nu}(\BC)}%{L^{2,\nu}(\BC;\BC)}
\newcommand{\Fbci}{\mathcal{F}^{2,\nu}_i(\BC)}
\newcommand{\Fbc}{\mathcal{F}^{2,\nu}(\BC)}
\newcommand{\FCh}{\mathcal{F}^{2,\frac\nu2}(\C)}
\newcommand{\FChT}{\mathcal{F}^{2,\frac\nu2}_{\BC}(\C)}
\newcommand{\Ldc}{L^{2,\sigma}_{\R}(\D)}
\usepackage[pdftex]{hyperref}

\begin{document}

%-------------------------------------------------------------------------
% editorial commands: to be inserted by the editorial office
%
%\firstpage{1} \volume{228} \Copyrightyear{2004} \DOI{003-0001}
%\seriesextraline{This is the Concrete Title of this Book\br H.E. R and S.T.C. W, Eds.}
%\firstpage{1}

\title[$\BC$--Segal-Bargmann and $\BC$--fractional Fourier transforms]
 {Bicomplex analogs of Segal-Bargmann and fractional Fourier transforms}

%----------Author 1
\author[A. Ghanmi]{Allal Ghanmi}
\address{A.-P.D.E.-G.S., CeReMAR,\\
          Department of Mathematics, P.O. Box 1014,  Faculty of Sciences, \\
          Mohammed V University in Rabat, \\
          Morocco}

\email{ag@fsr.ac.ma}
\email{zine0khalil@gmail.com}

%\thanks{This work was completed with the support of our \TeX-pert.}
%----------Author 2
\author[K. Zine]{Khalil Zine}
\address{A.-P.D.E.-G.S., CeReMAR,\\
          Department of Mathematics, P.O. Box 1014,  Faculty of Sciences, \\
          Mohammed V University in Rabat, \\
          Morocco}

\email{zine0khalil@gmail.com}


\date{November 06, 2018}

\begin{abstract}
 We consider and discuss some basic properties of the bicomplex analog of the classical Bargmann space. The explicit expression of the integral operator connecting the complex and bicomplex Bargmann spaces is also given. The corresponding bicomplex Segal--Bargmann transform is introduced and studied as well. Its explicit expression as well as the one of its inverse are then used to introduce a class of two--parameter bicomplex Fourier transforms (bicomplex fractional Fourier transform).
 This approach is convenient in exploring some useful properties of this bicomplex fractional Fourier transform.
\end{abstract}

%%% ----------------------------------------------------------------------
\maketitle
%%% ----------------------------------------------------------------------
%\tableofcontents

\section{Introduction} \label{s1}
The bicomplex (or Tetra) numbers (denoted here by $\BC$) are a special generalization of the complex numbers. Roughly speaking, they are complex numbers with complex coefficients. Since their introduction by Segre \cite{Segre1892}, different mathematical topics have been developed and investigated, including bicomplex functional analysis \cite{Dragoni,Spampinato1935a,Spampinato1935b,GLMRochon2011a,LunaShapiroStruppaVajiacBook2015}, bicomplex differentiability \cite{Price1991,AlpayLunaShapiroStruppa2014Book,LunaShapiroStruppaVajiacBook2015} and bicomplex quantum \cite{Davenport1978,RochonTremmblay2004,RochonTremmblay2006,
	GLMRochon2010AFA,MathieuMarchildonRochon2012}.
  For a complete treatment on bicomplex function theory as well as their applications see, e.g., \cite{Price1991,Ryan14,Ryan15,Rochon2008,ColomboSabadiniStruppaVajiacVajiac2011a,
  ColomboSabadiniStruppaVajiacVajiac2011b,AlpayLunaShapiroStruppa2014Book,LunaShapiroStruppaVajiacBook2015}.

   The main purpose of the present paper is to introduce and study in some detail the basic properties of interesting integral transforms in the framework of some infinite bicomplex Hilbert spaces. We begin by considering the bicomplex Bargmann space $\FBC$ and the bicomplex Segal--Bargmann transform $\BargBC$ in the context of the bicomplex holomorphic functions, the analogs of the classical ones $\mathcal{F}^{2,\gamma }(\C)$ and $\mathcal{B}^{\sigma,\gamma }_{\C}$ for holomorphic functions. The first result concerning $\FBC$ is obtained as immediate corollary of the key observation (Theorem
 \ref{MainThm1}). It decomposes $\FBC$ in terms of the classical ones with respect to the idempotent decomposition of bicomplex holomorphic functions.
    We next deal with the integral representations of the $L^2$--bicomplex holomorphic functions. Namely, Theorem \ref{MainThm41} expresses the fact that $\FBC$ is a reproducing kernel bicomplex Hilbert space.
     On the other hand, Theorem \ref{MainThm4} connects $\FCh$ to the special subspace of $\FBC$ leaving $\C_i:=\C+j\{0\}$ invariant, through a special explicit integral transform.
  This transform is further a surjection from the space of $L^2$--holomorphic functions on the complex plane $\C$ with values in $\BC$ onto the bicomplex Bargmann space $\Fbc$.
    Concerning the bicomplex Segal--Bargmann transform $\BargBC$, we show that it is a unitary isometric transform from the space of bicomplex-valued square integrable functions on the real line onto the bicomplex holomorphic Bargmann space (Theorem \ref{MainThm2}). 
   We give the explicit expression of the inverse $[\BargBC]^{-1}$ %and $[ \mathcal{A}^\nu ]^{-1}$
   (Theorem \ref{MainThm2inverse}). 
     Then, we use both $\BargBC$ and $[\BargBC]^{-1}$ to introduce a two--parameter family of bicomplex fractional Fourier transforms (BFrFT) labeled by the set of bicomplex numbers $\theta$; $|\theta|=1$, $\theta\ne \pm 1,\pm ij$, so that one recovers the classical $i$-Fourier transform as well as its variant; the $j$-Fourier transform. The basic properties of BFrFT, such as the uniqueness theorem, the Plancherel theorem as well as the inversion formula are obtained.

  The content of the paper can be described as follows.
  We collect in Section 2 some needed backgrounds on bicomplex numbers as well as on the bicomplex holomorphic functions.
  We also review the definition and the basic properties of infinite bicomplex Hilbert spaces and provide a basic example of them.
    In Section 3, we establish the main results concerning the bicomplex Bargmann space.
    We next present in Section 4 two integral reproducing properties of the $L^2$--bicomplex holomorphic functions.
    We devote Section 5 to the bicomplex Segal--Bargmann transform $\BargBC$ and to its basic properties.
    While Section 6 is concerned with a class of bicomplex fractional Fourier transforms.
 
\begin{rem}
In the present paper, we deal with the Hilbert spaces $L^{2,\alpha}(X)$ of all square integrable $\C$--valued functions on $X=\R,\C, \C^2$ with respect to the Gaussian measure ${c^\alpha_{d_X}}e^{-\alpha\norm{u}^2_X}d\lambda_X(u)$.
The analog Hilbert spaces on $\BC$ or with values in $\BC$ are appropriately defined in Sections 2, 3 and 4.
 The $d_X$ in $c_{d_X}^\alpha$ can be interpreted as the complex dimension of $X$,
  so that $d_X=0,1,2$ for $X=\R,\C,\C^2$ respectively.
The normalization $c^\alpha_{d_X}$ varies from one Hilbert space to another and is taken such that $\int_X c^\alpha_{d_X}e^{-\alpha\norm{u}^2_X}d\lambda_X(u)=1$. Thus, we have
$$ c^\alpha_{\BC} = c^\alpha_{2} =  4 c^{\alpha/2}_{2} = (c^\alpha_{1})^2 = 4  (c^{\alpha/2}_{1})^2 =(c^\alpha_{0})^4 =\left(\frac{\alpha}{\pi}\right)^2.$$
\end{rem}

    \section{Preliminaries}
    This section reviews the needed notions and results from the theory of bicomplex holomorphic functions and bicomplex Hilbert spaces. For more details, we refer the reader to \cite{Price1991,RochonShapiro2004,RochonTremmblay2004,RochonTremmblay2006}. We begin by recalling some basic definitions related to bicomplex numbers.

        \subsection{Bicomplex (or Tetra) numbers}
   They are a special generalization of complex numbers, $z=x+iy$; $x,y\in  \R$, $i^2=-1$, by means of entities specified by four real numbers.
   In abstract algebra, a bicomplex number is a pair $(z_1,z_2)$ of complex numbers constructed by the Cayley-Dickson process that defines the bicomplex conjugate $(z_1,-z_2)$.
 More precisely,
$$ \BC := \left\{  Z= z_1 + j z_2 ; \, z_1,z_2 \in \C \right\},$$
    where $j$ is a pure imaginary unit independent of $i$ such that $ij=ji$.
    Unlike the quaternions, the bicomplex numbers form a commutative algebra. % over $\C$.
    Addition and multiplication in $\BC$ are defined in a natural way.
    The complex conjugate of $Z= z_1 + j z_2\in \BC$ with respect to $j$ is given by
    $Z^\dag= z_1 - j z_2.$
    However, there are other forms of complex conjugates, to wit $\widetilde{Z}= \overline{z_1} + j \overline{z_2}$ and $Z^*= \overline{z_1} - j \overline{z_2}$.
     According to the nullity of $ ZZ^\dag =z_1^2+z_2^2$, we distinguish two interesting classes of bicomplex numbers.
   Indeed, $ZZ^{\dag}\ne 0$ characterizes those that are invertible. While $ZZ^{\dag} = 0$ is equivalent to $Z= \lambda (1\pm ij )$ for certain complex number $\lambda\in \C=\C_i$ and characterizes those that are zero divisors in $\BC$. Thus, one considers the idempotent elements
$$ \eu  = \frac{1+ij}2 \quad \mbox{and} \quad \es  = \frac{1-ij}2$$
   which satisfy the identities
   $$ \eu ^2=\eu , \quad \es ^2=\es , \quad \eu +\es  =1 , \quad \eu -\es  =ij, \quad \eu \es =0.$$
   The last identity shows in particular that $\BC$ is not a division algebra and that $\eu =\frac 12(1,i)$ and $\es =\frac 12(1,-i)$  are orthogonal with respect to the Euclidean inner product in $\C^2$.
   Thus, any $Z= z_1 + j z_2\in \BC$ can be rewritten in a unique way as
\begin{align}\label{ib}
 Z= (z_1 - i z_2) \eu  +  (z_1 + i z_2) \es  = \alpha \eu  +  \beta \es
\end{align}
  with $\alpha=z_1 - i z_2,\beta=z_1 + i z_2\in \C$.  Therefore, it is immediate to check that the $Z^\dag$-conjugate, the $\widetilde{Z}$-conjugate and the $Z^*$-conjugate read respectively $$Z^\dag=\beta \eu  + \alpha \es , \quad \widetilde{Z}=\overline{\beta}\eu  + \overline{\alpha}\es  \quad\mbox{and} \quad Z^*= \overline{\alpha}\eu  + \overline{\beta}\es .$$
   The idempotent representation \eqref{ib} is a central observation that simplifies considerably the computation with bicomplex numbers and reduces them to complex numbers.
   This reduction is possible thanks to the identity $\eu \es =0$.
   Thus, for given $Z=\alpha \eu  + \beta \es $ and $W=\alpha' \eu  + \beta' \es $, we have $ ZW = \alpha\alpha' \eu  + \beta \beta' \es  $ as well as 
   $Z^n = \alpha^n \eu  +  \beta^n \es  $.
    Accordingly, the exponential $e^Z$ of a bicomplex number $Z$, as specified by \eqref{ib}, can be defined by
$$e^Z = e^{\alpha \eu + \beta \es}= e^{\alpha} \eu  +  e^{\beta} \es .$$
    Further elementary functions are introduced and studied in \cite{RochonShapiro2004,AlpayLunaShapiroStruppa2014Book,LunaShapiroStruppaVajiacBook2015}.
    More details on some algebraic properties can be found in \cite{Price1991,RochonTremmblay2006}.

  \subsection{Bicomplex holomorphic functions}
   Following \cite{Price1991}, a $\BC$--valued function $f=f_1 +j f_2 $ on an open set $\Omega\subset \BC$ is said to be $\BC$-holomorphic at a point $Z_0\in\BC$ if it admits a bicomplex derivative at $Z_0$, i.e., if the limit
$$
\lim\limits_{\substack{ H \to 0 \\ H \notin \mathcal{NC}}} \frac{f(Z_0 + H) - f(Z_0)}{H}
$$
   exists and is finite, where $\mathcal{NC}$ denotes the null cone of bicomplex numbers defined by
   \begin{align*}
   \mathcal{NC}&=\{Z=z_1+jz_2 \in \BC; \, z_1^2+z_2^2=0\}\\&= \{Z=\alpha \eu  + \beta \es ; \, \alpha \beta = 0\}.
   \end{align*}
   This is equivalent to say that the $\C$--valued functions $f_1$ and $f_2$ are holomorphic in the variables $(z_1,z_2)$ with $Z=z_1+jz_2$ and satisfy the Cauchy-Riemann system
$$
\frac{\partial f_1}{\partial z_1} = \frac{\partial f_2}{\partial z_2}  \quad \mbox{and} \quad \frac{\partial f_1}{\partial z_2} = - \frac{\partial f_2}{\partial z_1} ,
$$
which we can rewrite in matrix representation as
$$
\left( \begin{array}{cc} \partial_{z_1}   & - \partial_{z_2}  \\ \partial_{z_2}   &  \partial_{z_1} \end{array}\right)
\left( \begin{array}{c} f_1  \\ f_2 \end{array}\right)= \left( \begin{array}{c} 0  \\ 0 \end{array}\right).
$$
   Here, the shorthand $\partial_z$ is used to mean the differential operator $\partial/\partial z$.
   The following characterization of $\BC$-holomorphicity is given in \cite{Rochon2008} and shows that bicomplex holomorphic functions are once again solutions of a linear system of differential equations with constant coefficients. Namely, a given $f\in \mathcal{C}^1(\Omega)$ is $\BC$-holomorphic on $\Omega$ if and only if $f$ satisfies the following three systems of differential equations
$$
\frac{\partial f}{\partial Z^{*}} = \frac{\partial f}{\partial Z^{\dag}}  = \frac{\partial f}{\partial \widetilde{Z}}  = 0,
$$
    where
$$
\frac{\partial}{\partial Z^{*}} = \frac{\partial }{\partial \overline {z_1}} + j  \frac{\partial }{\partial \overline {z_2}}   ;
\qquad \frac{\partial}{\partial Z^{\dag}} =  \frac{\partial }{\partial z_1} + j  \frac{\partial }{\partial z_2} ;
\qquad \frac{\partial f}{\partial \widetilde{Z}} = \frac{\partial }{\partial \overline {z_1}} - j  \frac{\partial }{\partial \overline {z_2}}  .
$$
   The above system is the foundation pillar for the theory of bicomplex holomorphic functions.
    Accordingly, any bicomplex holomorphic $\BC$--valued function $f$ is of the form \cite[Theorem 15.5]{Price1991}
\begin{align}\label{decomp:holom}
f(Z)=f(\alpha \eu  + \beta \es ) = \phi^+(\alpha) \eu  + \phi^-(\beta) \es ,
\end{align}
    where $\phi^\pm:\C \longrightarrow \C$ are holomorphic functions on $\C$.
     This is a key tool that we use in proving Theorem \ref{MainThm1} below. We denote by $\mathcal{BH}ol(\BC)$ the space of bicomplex holomorphic functions on $\BC$ taking their values in $\BC$.

\subsection{Infinite bicomplex Hilbert space}
  In order to define Hilbert space in the bicomplex setting, one needs to extend the notions of inner product and norm to the $\BC$-modules.
  Let $M$ be a $\BC$-module and consider the $\C$-vector spaces $V^+=M\eu$ and $V^- = M \es$.  Thus, $M$ can be seen as a $\C$-vector space by considering $M'=V^+\oplus V^-$. In general, $V^+$ and $V^-$ bear no structural similarities.

   According to \cite{RochonTremmblay2004}, an inner product on $M$ is a given functional $\scal{ \cdot, \cdot } : M\times M \longrightarrow \BC$ satisfying
  $\scal{\phi, \tau\psi + \varphi }  = \tau^{*} \scal{\phi,\psi} + \scal{\phi,\varphi }$
  for every $\tau\in \BC$ and $\phi, \psi , \varphi \in M$, and
  $ \scal{\phi, \psi } = \scal{\psi, \phi }^*$
  as well as $ \scal{\phi, \phi}=0$ if and only if $\phi=0$.
  It should be mentioned here that the projection $\scal{ \cdot, \cdot }_{V^\pm}$ of $\scal{ \cdot, \cdot }$ to each $V^\pm$ is a standard scalar product on $V^\pm$. More precisely, if $\phi, \varphi$ are in $ M $ identified to $ V^+ \oplus V^-$, we have
$$
\scal{ \phi, \varphi } =  \scal{\phi^+, \varphi^+}_{V^+} \eu   + \scal{\phi^-, \varphi^-}_{V^-} \es  ,
$$
  where $\psi^+:= \psi \eu \in V^+$ and $\psi^-:= \psi \es \in V^-$ for given $\psi\in M$.
  Notice that any $\BC$-scaler product on $M$ is completely determined in this way (see \cite[Theorem 2.6]{GLMRochon2010AFA}).
  A trivial example of a $\BC$-inner product on $M=\BC$ is the following
\begin{align}\label{spBC}
\scal{ Z,W}  = ZW^{*} = \alpha \overline{\alpha'} \eu  + \beta \overline{\beta'}\es  ,
\end{align}
  where  $ Z= \alpha \eu  + \beta \es $ and $W=\alpha' \eu  + \beta' \es $ are the idempotent representations of $Z$ and $W$ in $\BC$, respectively.

   A $\BC$-module $M$ is said to be a normed $\BC$-module if there exists a map $\norm{\cdot}: M\longrightarrow \mathbb{R}$ such that
\begin{enumerate}
\item $\norm{\cdot}$ is a norm on the vector space $V_{1} \oplus V_{2}$, and
\item $\norm{ \lambda \phi } \leq \sqrt{2} |\lambda| \norm{\phi}$ for all $\lambda \in \BC$ and all $\phi \in  M$.
\end{enumerate}
   $\norm{\cdot}$ is then called a $\BC$-norm on $M$.
   As in the theory of $\C$-vector spaces, a $\BC$-norm can always be induced from a $\BC$-scalar product by considering
\begin{align}\label{inducednorm}
\norm{\phi}^2 =  \frac{ \scal{\phi^+, \phi^+}_{V^+} + \scal{\phi^-, \phi^-}_{V^-} }{2}    = \left| \scal{\phi, \phi}\right|,
\end{align}
    where $\phi=\phi^+ + \phi^- $ and the modulus $|\cdot|$ denotes the usual Euclidean norm of $Z$ in $\mathbb{R}^{4}$ given by
\begin{align}\label{euclideannorm}
|Z|^{2}=|z_{1}|^{2}+|z_{2}|^{2}=\frac{|\alpha|^{2}+|\beta|^{2}}{2}
\end{align}
   for given $Z=z_1 + z_2 j=\alpha \eu + \beta \es $; $z_1, z_2,\alpha,\beta \in \C$. The Euclidean norm in \eqref{euclideannorm} is in fact the one induced from the $\BC$-scaler product in \eqref{spBC}.
   The norm in \eqref{inducednorm} obeys a generalized Schwarz inequality (\cite[Theorem 3.7]{GLMRochon2010AFA})
\begin{align}\label{SchIneq}
\left|\scal{\phi, \varphi}\right| \leq \sqrt{2} \norm{\phi} \norm{\varphi} .
\end{align}
   Accordingly, the concept of infinite bicomplex Hilbert space was defined in \cite{GLMRochon2010AFA}.
   It is a $\BC$-inner product space $(M, \scal{\cdot, \cdot})$ which is complete with respect to the induced $\BC$-norm \eqref{inducednorm}.
   The next result is interesting in itself and its proof is contained in Theorems 3.4, 3.5 and Corollary 3.6 in \cite{GLMRochon2010AFA}.

\begin{thm}\label{thmBCIHS}
The $\BC$-module $(M, \scal{\cdot, \cdot} )$ is an infinite bicomplex Hilbert space if and only if
$(V^\pm, \scal{\cdot, \cdot}_{V^\pm})$ are $\C$-Hilbert spaces.
\end{thm}

 We conclude this subsection by recalling the bicomplex version of the classical Riesz' representation theorem \cite[Theorem 3.7]{GLMRochon2010AFA}.

\begin{thm}
 For every continuous $\BC$--valued linear functional $f$ on a $\BC$-module $M$, there exists a unique $\phi_0\in M$ such that for every $\phi\in M$ we have
\begin{align} \label{RieszRepThm}
f(\phi) = \scal{ \phi , \phi_0 } .
\end{align}
\end{thm}

 \subsection{Basic example}
  In order to provide an interesting example of infinite bicomplex Hilbert space, we need to fix further notation. Let  $L^{2,\nu}(X)$ be as in the end of the introductory section. We also consider the space $\Hbc$ of all $\BC$--valued Borel measurable functions $f$ on $\BC$ subject to $\norm{f}_{\Hbc} <+\infty$. The bicomplex norm $\norm{f}_{\Hbc}$ is the one induced from the bicomplex inner product
$$ \scal{ f,g}_{\Hbc}:=c_{\Hbc}^\nu \int_{\BC} \scal{ f(Z),g(Z)} e^{-\nu |Z|^{2}}  d\lambda(Z) $$
via \eqref{inducednorm}, where $\scal{\cdot ,\cdot}$, in the integrand, is the standard bicomplex inner product on $\BC$ defined by \eqref{spBC}.
In fact, for every $f=f_1 \eu + f_2\es  ,g= g_1 \eu + j g_2 \es$,  we have 
$$\scal{f,g}_{\Hbc} = \scal{f_1,g_1}_{L^{2,\nu}(\C^2)} \eu + \scal{f_2,g_2}_{L^{2,\nu}(\C^2)} \es,$$
where $f_k$ and $g_k$; $k=1,2$, are seen as $\C$--valued functions on $ \C^2 $ in the variables $(z_1,z_2)$. Thus, we have
\begin{align}\label{normL2}
\norm{f}_{\Hbc}^{2}= \frac{1}{2}  \left( \norm{f_1}_{L^{2,\nu}(\C^2)}^{2} + \norm{f_2}_{L^{2,\nu}(\C^2)}^{2} \right) .
\end{align}
Subsequently, the following decomposition, with respect to the variables $z_1$ and $z_2$,
\begin{align} \label{DecompHilbert1}
\Hbc = L^{2,\nu}(\C^2) \eu  +  L^{2,\nu}(\C^2) \es,
\end{align}
readily follows from \eqref{normL2}.
Another interesting decomposition of $\Hbc$ with respect to the idempotent representation of bicomplex numbers is the following

\begin{prop}\label{DecompHilbert} We have
\begin{align} \label{DecompHilbert2}
\Hbc = L^{2,\frac{\nu}{2}}(\C^2) \eu  +  L^{2,\frac{\nu}{2}}(\C^2) \es .
\end{align}
More precisely, for every $f\in \Hbc$ there exist some $\phi^{\pm} \in L^{2,\frac{\nu}{2}}(\C^2)$ such that
$$f(\alpha \eu+\beta \es)=\phi^+(\alpha,\beta)\eu+\phi^-(\alpha,\beta)\es$$
with 
\begin{align} \label{norm-bc2}
\norm{f}_{\Hbc}^2 &=  \frac{1}{2} \left( \norm{\phi^+}^{2}_{L^{2,\frac{\nu}{2}}(\C^2)} + \norm{\phi^-}^{2}_{L^{2,\frac{\nu}{2}}(\C^2)} \right).
\end{align}
Moreover, the space $\Hbc$ is an infinite bicomplex Hilbert space on $\BC$.
\end{prop}

 \begin{proof}
 For any $f= f_{1} + j f_{2}$ and $g= g_{1} + j g_{2} \in \Hbc$, we consider the functions $\phi^{\pm}=f_{1}\mp i f_{2}$ and $\varphi^{\pm}=g_{1}\mp i g_{2}$, seen as $\C$--valued functions on $ \C^2 $ in the variables $(\alpha,\beta)$, so that
$$f(\alpha \eu+\beta \es)=\phi^+(\alpha,\beta)\eu+\phi^-(\alpha,\beta)\es$$ and
$$g(\alpha \eu+\beta \es)=\varphi^+(\alpha,\beta)\eu+\varphi^-(\alpha,\beta)\es.$$
   Thus, we have
\begin{align} \label{DecompScaler2}
\scal{f,g}_{\Hbc} =  \scal{\phi^+,\varphi^+}_{L^{2,\frac{\nu}{2}}(\C^2)} \eu + \scal{\phi^-,\varphi^-}_{L^{2,\frac{\nu}{2}}(\C^2)} \es ,
\end{align}
since the Lebesgue measure on $\BC\equiv \R^4$ reads
$ d\lambda(Z) =   dx_{1}dy_{1}dx_{2}dy_{2} = \frac{1}{4}  d\lambda(\alpha)  d\lambda(\beta).$
  Therefore, \eqref{normL2} reduces further to \eqref{norm-bc2}.
This completes our check of \eqref{DecompHilbert2}.
  Finally, the result that $\Hbc$ is an infinite bicomplex Hilbert space on $\BC$ readily follows making use of Theorem \ref{thmBCIHS} thanks to \eqref{DecompHilbert2} (or also \eqref{DecompHilbert1}).
\end{proof}

 \section{$\BC$-Bargmann space} 
 Recall first that the classical Bargmann space consists of all holomorphic functions on the complex plane subject to norm boundedness with respect to the Gaussian measure with given normalization constant $c^{\gamma}_1$,
   \begin{align}\label{cBFs}
\mathcal{F}^{2,\gamma }(\C) := Hol(\C) \cap L^{2}\left(\C; c^{\gamma}_1 e^{-\gamma  |z|^2}dxdy\right) ;  
\gamma >0 .
\end{align}
It  is well--studied in the literature \cite{Bargmann1961,Bargmann1962,Segal1963} and is a convenient setting for many problems in functional analysis, mathematical physics, and engineering. Basic references in these areas are e.g. \cite{Bargmann1961,Folland1989,Hall2000,Zhu2012} and the references therein.

The bicomplex counterpart of the classical Bargmann space in \eqref{cBFs} is defined to be the subspace of the infinite Hilbert space $\Hbc$ consisting of bicomplex holomorphic functions on $\BC$,
\begin{align*}
\mathcal{F}^{2,\nu}(\BC)=\Hbc \cap \mathcal{BH}ol(\BC).
\end{align*}
The following result shows that $\mathcal{F}^{2,\nu}(\BC)$ has a Hilbertian decomposition
 with respect to the $(\alpha,\beta)$-idempotent variables. 

\begin{thm}\label{MainThm1}
  We have
\begin{align} \label{DecompBFS}
\mathcal{F}^{2,\nu}(\BC)=\mathcal{F}^{2,\frac{\nu}{2}}(\C)\eu +\mathcal{F}^{2,\frac{\nu}{2}}(\C)\es .
\end{align}
 Succinctly, every $f\in \mathcal{F}^{2,\nu}(\BC) $ can be re-expressed as
$$
f(\alpha \eu +\beta \es )= \phi^{+}(\alpha)\eu + \phi^{-}(\beta)\es .
$$
The $\phi^{\pm}$ are $\C$--valued functions belonging to the classical Bargmann space
$\mathcal{F}^{2,\frac{\nu}{2}}(\C)$. % := Hol(\C) \cap L^{2,\frac\nu2}(\C) $.
  Moreover, we have
%\begin{align} \label{norm-bc2f}
%\norm{f}_{\Hbc}^2 &=  \frac{\pi   c_{\Hbc}^\nu}{4\nu  c_1^{\frac\nu2} }
% \left( \norm{\phi^+}^{2}_{L^{2,\frac{\nu}{2}}(\C)} + \norm{\phi^-}^{2}_{L^{2,\frac{\nu}{2}}(\C)} \right).
%\end{align}
\begin{align} \label{norm-bc2f}
\norm{f}_{\Hbc}^2 &=  \frac{1}{2}
 \left( \norm{\phi^+}^{2}_{L^{2,\frac{\nu}{2}}(\C)} + \norm{\phi^-}^{2}_{L^{2,\frac{\nu}{2}}(\C)} \right).
\end{align}
\end{thm}

\begin{proof}
   By \eqref{DecompHilbert2}, every $\BC$--valued function $f$ on $\BC$ is of the form
   $f(Z)= f(\alpha \eu+\beta \es)=\phi^+(\alpha,\beta)\eu+\phi^-(\alpha,\beta)\es$,
   where the $\C$--valued functions $\phi^{\pm}(\alpha,\beta)$ on $\C^2$ belong to $L^{2,\frac\nu2}(\C^2)$ and satisfy
    \eqref{norm-bc2}. Therefore, $\phi^{\mp} \in \mathcal{F}^{2,\frac\nu2}(\C^2)$.
   On the other hand,  Theorem 15.5 in \cite[p.87]{Price1991} shows that the function $\phi^{+}$ (resp. $\phi^{-}$) is a $\C$--valued holomorphic function on $\C^{2}$ that depends only in $\alpha$ (resp. $\beta$).
   Next, by Fubini's theorem and the fact $  c_1^{\frac{\nu}{2}}\int_{\C} e^{-\frac{\nu}{2}|\xi|^{2}}d\lambda(\xi)= 1$, we obtain
\begin{align*}
\norm{\phi^{+}}_{L^{2,\frac\nu2}(\C^2)}^2 &= c_2^{\frac{\nu}{2}} \int_{\C^2} \left| \phi^{+}(\alpha) \right|^{2}  e^{-\frac{\nu}{2}(|\alpha|^{2}+|\beta|^{2})} d\lambda(\alpha)d\lambda(\beta)
\\&= c_2^{\frac{\nu}{2}} \left( \int_{\C} e^{-\frac{\nu}{2}|\beta|^{2}} d\lambda(\beta) \right) \left(\int_{\C} \left| \phi^{+}(\alpha) \right|^{2} e^{-\frac{\nu}{2}|\alpha|^{2}}  d\lambda(\alpha)\right)
\\& =    \norm{\phi^{+}}_{L^{2,\frac\nu2}(\C)}^2.
\end{align*}
Thus, the condition $\norm{ \phi^{+}}_{L^{2,\frac\nu2}(\C^2)}< +\infty$ implies that $\phi^{+} \in \mathcal{F}^{2,\frac{\nu}{2}}(\C)$. Similarly, we have $\norm{\phi^{-}}_{L^{2,\frac\nu2}(\C^2)}^2  =    \norm{\phi^{-}}_{L^{2,\frac\nu2}(\C)}^2$ and therefore $\phi^{-} \in \mathcal{F}^{2,\frac{\nu}{2}}(\C)$. Accordingly, \eqref{norm-bc2} reduces further to \eqref{norm-bc2f}. 
\end{proof}

\begin{rem}
This is a special way to create two models of Bargmann spaces to work with simultaneously. The first one is focused on $e_+$ and the other on $e_-$.
\end{rem}

\begin{rem} \label{remspHbc}
 For bicomplex holomrphic functions $f,g: \BC \longrightarrow \BC$ with $f=\phi^+\eu +\phi^-\es$ and  $g=\varphi^+\eu +\varphi^-\es$, the quantity in \eqref{DecompScaler2} reduces further to
 \begin{align} \label{spHbc}
\scal{f,g}_{\Hbc} = \scal{\phi^+,\varphi^+}_{L^{2,\frac{\nu}{2}}(\C)} \eu + \scal{\phi^-,\varphi^-}_{L^{2,\frac{\nu}{2}}(\C)} \es .
\end{align}
\end{rem}

\begin{cor}
  The space $\mathcal{F}^{2,\nu}(\BC)$ is an infinite $\BC$-Hilbert space.
\end{cor}

\begin{proof}
  This is an immediate consequence of Theorem \ref{MainThm1} and Theorem \ref{thmBCIHS}. 
\end{proof}

\begin{cor}
   A function $f$ belongs to $\mathcal{F}^{2,\nu}(\BC)$ if and only if $f$ can be expanded as $f(Z)=\sum\limits_{n=0}^{\infty} A_{n}Z^{n}$ for some bicomplex sequence $(A_n)_n$ satisfying the growth condition
\begin{align}  \label{normfseries}
%\norm{f}_{\mathcal{F}^{2,\nu}(\BC)}^2 =
  \sum_{n=0}^{\infty}  \frac{2^n n!}{\nu^n}   \left|A_{n}\right|^{2} < +\infty.
\end{align}
\end{cor}

\begin{proof}
In view of Theorem \ref{MainThm1}, the functions $\phi^{+}$ and $\phi^{-}$ belong to $\mathcal{F}^{2,\frac{\nu}{2}}(\C)$, and therefore they can be expanded in power series as $\phi^{+}(\alpha)=\sum\limits_{n=0}^{\infty} a_{n}^{+}\alpha^{n}$ and $\phi^{-}(\beta)=\sum\limits_{n=0}^{\infty}a_{n}^{-}\beta^{n}$,
where the complex sequences $(a_{n}^{\pm})_n$ are subject to the growth conditions
\begin{align*}
 \norm{\phi^{\pm}}_{L^{2,\frac\nu2}(\C)}^2 =  \sum_{n=0}^{\infty} \left(\frac{2}{\nu}\right)^n  n!\left|a_{n}^{\pm}\right|^{2}<\infty .
 \end{align*}
Subsequently, we get 
 \begin{align*}
  f(Z) &       = \sum_{n=0}^{\infty}(a_{n}^{+} \eu + a_{n}^{-} \es )(\alpha^{n}\eu +\beta^{n}\es )
      =\sum_{n=0}^{\infty} A_{n}Z^{n},
 \end{align*}
where $A_{n}:= a_{n}^{+} \eu + a_{n}^{-}\es \in \BC$. Next, from \eqref{norm-bc2f}, one obtains the growth condition \eqref{normfseries}.
Indeed, we have 
\begin{align*}
\norm{f}_{\Hbc}^2 &=  \frac{1}{2}
 \sum_{n=0}^{\infty} \left(\frac{2}{\nu}\right)^n  n! \left(\left|a_{n}\right|^{2}+\left|b_{n}\right|^{2} \right)
=  \sum_{n=0}^{\infty}  \left(\frac{2}{\nu}\right)^n  n! \left|A_{n}\right|^{2} .
\end{align*}
This completes the proof.
\end{proof}

\section{Two integral representations of $L^2$-bicomplex holomorphic functions}

The first result in this section is a consequence of Theorem \ref{MainThm1}.

\begin{thm}\label{MainThm41} The space $\mathcal{F}^{2,\nu}(\BC)$ is a reproducing kernel infinite $\BC$-Hilbert space whose reproducing kernel is given by
\begin{align}\label{explicirRHbc}
K_{\BC}^{\nu}\left( Z,W\right) =  e^{\frac\nu 2 Z W^{*}}.
\end{align}
Succinctly, for every $f\in \mathcal{F}^{2,\nu}(\BC)$, we have
\begin{align}\label{REpKerbc}
f(Z)  = c_{\Hbc}^\nu \int_{\BC} e^{\frac{\nu}{2} Z W^{*}} f(W)  e^{-\nu|W|^2} d\lambda(W).
\end{align}
\end{thm}

\begin{proof}
The fact that $\mathcal{F}^{2,\nu}(\BC)$ is a reproducing kernel infinite $\BC$-Hilbert is an immediate consequence of Theorem \ref{MainThm1} and the $\BC$-Riesz' representation theorem \eqref{RieszRepThm}, since for every $f \in$  $\mathcal{F}^{2,\nu}(\BC)$, we have 
\begin{align*} 
|f(Z)| \leq \sqrt{2} |e^{\frac{\nu}{4}Z Z^*}| \norm{f}_{\mathcal{F}^{2,\nu}(\BC)}.
\end{align*}
This can be handled making use of the generalized $\BC$-Schwarz inequality \eqref{SchIneq} and the expression \eqref{normfseries} giving the norm $\norm{f}_{\mathcal{F}^{2,\nu}(\BC)}$.
The explicit expression \eqref{explicirRHbc} of the reproducing kernel $K_{\BC}^{\nu}\left( Z,W\right)$ is obtained by only proving the reproducing property \eqref{REpKerbc} thanks the uniqueness of the reproducing kernel. To this end, denote the right-hand side of \eqref{REpKerbc} by
$\mathcal{P}^\nu f (Z)$,
$$\mathcal{P}^\nu f (Z) = \scal{f, K_{\BC}^{\nu}\left( \cdot , Z \right)}_{\Hbc},$$
and notice that the restriction of $K_{\BC}^{\nu}$ in \eqref{explicirRHbc} to $\C \times \C=(\C+j\{0\})\times (\C+j\{0\})$ reduces further to the reproducing function $K_{\C}^{\frac{\nu}{2}}$ of the Bargmann space $\mathcal{F}^{2,\frac{\nu}{2}}(\C)$, so that
\begin{align}\label{RepKerc}
 \phi (z) = \scal{ \phi, K_{\C}^{\frac{\nu}{2}}(\cdot,z) }_{L^{2,\frac{\nu}{2}}(\C)} =  c_1^{\frac{\nu}{2}}  \int_{\C} e^{\frac{\nu}{2} z\overline{\xi}} \phi (\xi) e^{-\frac{\nu}{2}|\xi|^{2}}d\lambda(\xi)
\end{align}
 for every $\phi \in \mathcal{F}^{2,\frac{\nu}{2}}(\C)$. Moreover, in virtue of Remark \ref{remspHbc} 
 as well as the facts
$K_{\BC}^{\nu}(Z,W)= (K_{\BC}^{\nu}(W,Z))^*$ and
$$K_{\BC}^{\nu}\left( Z,W\right) = K_{\C}^{\frac{\nu}{2}}(\alpha,\xi) \eu + K_{\C}^{\frac{\nu}{2}}(\beta,\zeta) \es ,$$
with $Z=\alpha \eu + \beta \es$ and $W=\xi \eu + \zeta \es$,
we obtain
\begin{align*}
\mathcal{P}^\nu f (Z)  &=   \scal{f, K_{\BC}^{\nu}\left( \cdot , Z \right)}_{\Hbc}
\\ &
=  \scal{ \phi^+, K_{\C}^{\frac{\nu}{2}}(\cdot,\xi) }_{L^{2,\frac{\nu}{2}}(\C)} \eu + \scal{\phi^-, K_{\C}^{\frac{\nu}{2}}(\cdot,\zeta)}_{L^{2,\frac{\nu}{2}}(\C)} \es
  %\right)
.
\end{align*}
Therefore, the desired result \eqref{REpKerbc} follows making use of the reproducing property \eqref{RepKerc} for $\phi^{\pm}$ belonging to the classical Bargmann space $\mathcal{F}^{2,\frac{\nu}{2}}(\C)$.
\end{proof}

\begin{rem}
The explicit expression of the corresponding reproducing kernel $K_{\BC}^{\nu}$ can also be obtained using the expansion formula
$\displaystyle K_{\BC}^{\nu}(Z,W)= \sum_{n=0}^{\infty} \phi_n(Z)\phi^{*}_{n}(W),$
where $\phi_{n}$ represents any orthonormal basis of $\Hbc$ like the one provided in Lemma \ref{lem:basis} below.
\end{rem}

\begin{lem} \label{lem:basis}
The set of functions
\begin{align} \label{bcbasis} \phi_n(Z) := \left(\frac{\nu^n}{2^n n!  }\right)^{\frac{1}{2}}  Z^{n}
\end{align}
is a Schauder orthonormal basis of the infinite $\BC$-Hilbert space $\mathcal{F}^{2,\nu}(\BC)$.
\end{lem}

\begin{proof}
By means of Theorem \ref{MainThm1} and Remark \ref{remspHbc} combined with the fact $\eu  + \es=1$, it is clear that the monomials $Z^n =\alpha^n \eu  + \beta^n \es $ form an orthogonal basis of
$\mathcal{F}^{2,\nu}(\BC)$, for the monomials $e_n(\alpha)=\alpha^n$ form an orthogonal basis of $\mathcal{F}^{2,\frac{\nu}{2}}(\C)$.
More precisely, we have
$$\scal{ E_n,E_m}_{\Hbc}=   \scal{ e_n,e_m}_{L^{2,\frac{\nu}{2}}(\C)} =
 \frac{2^n n!}{\nu^n} \delta_{n,m}.$$
For the density of the monomials $Z^n$ in $\mathcal{F}^{2,\nu}(\BC)$, let $f\in \mathcal{F}^{2,\nu}(\BC)$ such that $\scal{f,E_n}_{\Hbc} = 0$ for every nonnegative integer $n$.
By expanding $f$ in power series as $ f(Z) =\sum\limits_{n=0}^\infty A_{n}Z^{n},$ we get
\begin{align*}
\scal{f ,E_n}_{\Hbc}  =\scal{\sum_{m=0}^\infty A_{m}E_m, E_n}_{\Hbc}  =   \frac{2^n n!}{\nu^n} A_{n}.
\end{align*}
This implies that $A_{n}=0$ for all $n$ and consequently $f$ is identically zero on $\BC$.
\end{proof}

We provide below another interesting integral representation of the elements of the bicomplex Bargmann space $\mathcal{F}^{2,\nu}(\BC)$ by means of their restriction to $\C+j\{0\}$. Let $ \mathcal{S}^\nu$ be the integral transform
\begin{align}\label{TransfFcHbc} \mathcal{S}^\nu (F)(Z) : = c_1^{\frac\nu2}%\left(\frac{\nu}{2\pi}\right)
 \int_\C F(\xi) e^{-\frac \nu2 Z\overline{\xi} - \frac{\nu}{2} |\xi|^2} d\lambda(\xi).
\end{align}
It is well--defined on the classical Bargmann space $\mathcal{F}^{2,\frac{\nu}{2}}(\C)$ of weight $\frac\nu2$, and connects it to the special subspace $\Fbci$ consisting of $f \in \Fbc$ leaving $\C_i:=\C+j\{0\}$ invariant, i.e., $f(\C+j\{0\})\subset \C+j\{0\}$. The space $\FCh$, of $L^2$ holomorphic functions on the complex plane $\C$ with values in $\BC$, is then a particular subspace of $\FChT$, $\FCh \subset \FChT$. Thus, we prove the following

\begin{thm} \label{MainThm4}
The transform $ \mathcal{S}^\nu $ maps $\FChT$ onto $\Fbc$ and $\FCh$ onto $\Fbci$.
 Moreover, we have $ \mathcal{S}^\nu (e_n)(Z) = Z^n$.
\end{thm}

\begin{proof}
Starting from the idempotent decomposition $f(Z) =\phi^{+}(\alpha)\eu +\phi^{-}(\beta)\es $ in  $\mathcal{F}^{2,\nu}(\BC)$ and using \eqref{RepKerc} for $\phi^{\pm} \in \mathcal{F}^{2,\frac{\nu}{2}}(\C)$, we see that $f$ admits the representation
\begin{align*}
f(Z) &= c_1^{\frac\nu2} %\frac{\nu}{2\pi}
\left( \eu \int_{\C} e^{\frac{\nu}{2} \alpha \overline{\xi}} \phi^+ (\xi) e^{-\frac{\nu}{2}|\xi|^{2}}d\lambda(\xi)
+ \es \int_{\C} e^{\frac{\nu}{2} \beta \overline{\xi}} \phi^- (\xi) e^{-\frac{\nu}{2}|\xi|^{2}}d\lambda(\xi)\right)
\\ &= c_1^{\frac\nu2} 
\int_{\C}
\left( e^{\frac{\nu}{2} \alpha \overline{\xi}} \eu +  e^{\frac{\nu}{2} \beta \overline{\xi}} \es \right)
\left(  \phi^+ (\xi) \eu  + \phi^- (\xi) \es \right) e^{-\frac{\nu}{2}|\xi|^{2}}d\lambda(\xi)
\\ &=c_1^{\frac\nu2}%\frac{\nu}{2\pi}
\int_{\C+j\{0\}} e^{\frac{\nu}{2} Z \overline{\xi}} f|_{\C+j\{0\}} e^{-\frac{\nu}{2}|\xi|^{2}}d\lambda(\xi)\\
&= \mathcal{S}^\nu (f|_{\C+j\{0\}})(Z).
 \end{align*}
 The function $f|_{\C+j\{0\}}$ is clearly holomorphic on $\C$ but with coefficients in $\BC$.
 This proves that $ \mathcal{S}^\nu: \FChT \longrightarrow \Fbc $ and its restriction $ \mathcal{S}^\nu|_{\FCh }: \FCh \longrightarrow \Fbci $
  are onto. A direct computation shows that
 the action of $ \mathcal{S}^\nu $ on the monomials $e_n(\xi)=\xi^n$ is given by
  \begin{align*}
 \mathcal{S}^\nu (e_n)(Z)
=c_1^{\frac\nu2}  \int_{\C} \xi^n \left( e^{\frac{\nu}{2} \alpha \overline{\xi}} \eu +  e^{\frac{\nu}{2} \beta \overline{\xi}} \es \right)
 e^{-\frac{\nu}{2}|\xi|^{2}}d\lambda(\xi)
= Z^n ,
 \end{align*}
 since $ \int_{\C} \xi^n e^{\gamma \alpha \overline{\xi}}  e^{-\gamma|\xi|^{2}}d\lambda(\xi)
=  \left(\frac{\pi}{\gamma}\right) \alpha^n$. %reproducing property for $\FCh$.
\end{proof}

\begin{rem} \label{MainThm4}
We have
 $$ \norm{\mathcal{S}^\nu (e_n)}_{\Hbc}^2 =  \norm{e_n}_{L^{2,\frac{\nu}{2}}(\C)}^2.$$
\end{rem}

  \begin{rem}
The space $\Fbci$ can be identified to the special phase subspace of $\mathcal{F}^{2,\nu}(\C^2)$ defined by
\begin{equation}\label{Image}
\mathcal{A}^{2,\nu}(\C^2) := \left\{f\in \mathcal{F}^{2,\nu}(\C^2); \, \left(\frac{\partial}{\partial z} + i \frac{\partial}{\partial w}\right) f =0 \right\}
\end{equation}
and characterized as the image of $L^{2,\sigma}(\R)$ by the special one-to-one transform $\mathcal{B}^{2,\nu}\circ \mathcal{B}^{1,\nu}$ obtained as the composition operator of the $1d$ and $2d$ Segal--Bargmann transforms \cite[Theorem 2.2]{BDG2018}. It can also be identified to the space of slice (left) regular functions on the quaternions leaving invariant the slice $\C_i\simeq \C$ in the quaternion $\Hq$ \cite[Theorem 3.4]{BDG2018}.
 \end{rem}

\section{The bicomplex Segal--Bargmann transform}

It is a known fact that the classical Bargmann space $\mathcal{F}^{2,\gamma }(\C)$ is unitary isomorphic to the quantum mechanical configuration space $L^{2,\sigma}(\R)$, of all $\C$--valued $e^{-\sigma x^2} dx$-square integrable functions on the real line, by considering the rescaled Segal--Bargmann transform (\cite{Bargmann1961,Segal1963,Folland1989,Zhu2012}) 
\begin{align} \label{defQSBT}
\mathcal{B}^{\sigma,\gamma}_{\C} (\psi)(z)  = c^{\sigma}_0 \int_{\R} e^{-\sigma  \left(x-\sqrt{\frac{\gamma}{2\sigma}}\, z \right)^2 }\psi(x)dx.
\end{align}
  Mathematical theory of Segal--Bargmann transform has interesting applications in many fields of mathematics and physics and is an essential tool in signal processing.  

We propose in the present section a bicomplex analog of $ \mathcal{B}^{\sigma,\gamma}_{\C}$ in \eqref{defQSBT} and study some of its basic properties. In fact, by Theorem \ref{MainThm1}, we can split any $f\in  \mathcal{F}^{2,\nu}(\BC)$ as
$$f(\alpha \eu +\beta \es )= \phi^{+}(\alpha)\eu + \phi^{-}(\beta)\es  $$
for some $\phi^{\pm}\in \mathcal{F}^{2,\frac{\nu}{2}}(\C)$.
Then, by means of \eqref{defQSBT},  there exist some $\varphi^+, \varphi^- \in L^{2,\sigma}(\R)$ such that
\begin{align}
f(Z) & =  \mathcal{B}^{\sigma,\frac{\nu}{2}}_{\C} (\varphi^+)(\alpha) \eu  +  \mathcal{B}^{\sigma,\frac{\nu}{2}}_{\C} (\varphi^-) (\beta)\es \label{decomSBT}
\\& =  \int_{\R}\left(  B^{\sigma,\frac{\nu}{2}}_{\C}(x;\alpha)  \eu  +   B^{\sigma,\frac{\nu}{2}}_{\C}(x;\beta)   \es \right)    \left( \varphi^+(x) \eu  + \varphi^-(x) \es \right) dx,  \nonumber
\\& = \int_{\mathbb{R}}  {B^{\sigma,\nu}_{\BC}}(x;Z) \varphi(x) dx , \label{SBT1}
\end{align}
where we have set $\varphi(x) := \varphi^+(x) \eu  + \varphi^-(x) \es $.
The function $\varphi:\R \longrightarrow \BC$ belongs to the bicomplex Hilbert space $L^{2,\sigma}_{\BC}(\R)$ consisting of all $e^{-\sigma x^2} dx$-square integrable $\BC$--valued functions whose component functions, with respect to the idempotent decomposition, belong to $L^{2,\sigma}(\R)$. The $\BC$-norm in $L^{2,\sigma}_{\BC}(\R)$ is the one associated to the bicomplex inner product
\begin{align}\label{spR1}
\scal{\varphi,\psi}_{L^{2,\sigma}_{\BC}(\R)} : =  c_0^{\sigma} \int_{\R}  \varphi(x) \psi(x)^{*} e^{-\sigma x^2}dx.
\end{align}
  The kernel function ${B^{\sigma,\nu}_{\BC}}(x;Z):= B^{\sigma,\frac{\nu}{2}}_{\C}(x;\alpha)  \eu  +   B^{\sigma,\frac{\nu}{2}}_{\C}(x;\beta)   \es$
  is explicitly given by
\begin{align} \label{KernelSBT1}
{B^{\sigma,\nu}}_{\BC}(x;Z)
= c^{\sigma}_0  e^{-\sigma  \left(x-\sqrt{\frac{\nu}{4\sigma}}\, Z \right)^2 }
\end{align}
and can be seen as the natural extension of $B^{\sigma,\frac{\nu}{2}}(x;\xi) $ to the bicomplex (holomorphic) setting in the second variable, $Z=\alpha  \eu + \beta  \es$.
The corresponding integral transform acts on $L^{2,\sigma}_{\BC}(\R)$ as defined by the left-hand side of \eqref{SBT1}, to wit
\begin{align}\label{BC:SBT}
{\mathcal{B}^{\sigma,\nu}_{\BC}} (\varphi )(Z) := c_0^{\sigma} \int_{\R}  e^{-\sigma \left(x-\sqrt{\frac{\nu}{4\sigma}}\, Z \right)^2 } \varphi(x)  dx,
\end{align}
 provided that the integral exists.

\begin{defn}
We call ${\mathcal{B}^{\sigma,\nu}_{\BC}}$ the bicomplex Segal--Bargmann transform.
\end{defn}

The action of ${\mathcal{B}^{\sigma,\nu}_{\BC}}$ on the rescaled real Hermite polynomials
 \begin{align}\label{HermitPol}
  H_n^{\sigma}(x) := (-1)^n e^{\sigma x^2} \frac{d^n}{dx^n}\left(e^{-\sigma x^2}\right)
   \end{align}
is given by
 \begin{align}\label{SBT0Hn}
 {\mathcal{B}^{\sigma,\nu}_{\BC}} (H_n^{\sigma})(\xi)
 = \frac{\norm{H_n^{\sigma}}_{L^{2,\sigma}(\R)}}{\norm{e_n}_{L^{2,\frac \nu 2}(\C)}} Z^n
  = \frac{\norm{H_n^{\sigma}}_{L^{2,\sigma}_{\BC}(\R)}}{\norm{E_n^{\nu}}_{L^{2,\nu}(\BC)}} Z^n.
 \end{align}
 It follows by means of \eqref{decomSBT} combined with the facts that 
 $$  \mathcal{B}^{\sigma,\frac{\nu}{2}}_{\C} (H_n^{\sigma})(\xi)  = \frac{\norm{H_n^{\sigma}}_{L^{2,\sigma}(\R)}}{\norm{e_n}_{L^{2,\frac \nu 2}(\C)}}  \xi^n \quad \mbox{ and } \norm{E_n}_{\Hbc}=\norm{e_n}_{L^{2,\frac{\nu}{2}}(\C)} .$$ Notice that we also have made use of
$\norm{H_n^\sigma}_{L^{2,\sigma}_{\BC}(\R)}=\norm{H_n^\sigma}_{L^{2,\sigma}(\R)}$ for $H_n^\sigma$ being real--valued.
Therefore, the considered transform ${\mathcal{B}^{\sigma,\nu}_{\BC}}$ maps the orthonormal basis
\begin{equation}\label{Orthonbasis}
\psi_n^{\sigma}(x) =    \left( \frac{1}{2^n \sigma^n n!} \right)^{1/2} H_n^{\sigma}(x),
\end{equation}
of the configuration space $L^{2,\sigma}_{\BC}(\R)$, to the monomials $ \phi_n(Z)$ in \eqref{bcbasis}, which form an orthonormal basis of $\Fbc$.
Using similar arguments, we can prove that
\begin{align}
\norm{{\mathcal{B}^{\sigma,\nu}_{\BC}} \varphi}_{\Hbc}^2 =  \norm{\varphi}^{2}_{L^{2,\sigma}_{\BC}(\R)}
\end{align}
for every $\varphi \in L^{2,\sigma}_{\BC}(\R)$ such that $f= \BargBC \varphi$.
Indeed,
%\begin{proof}
the result readily follows by means of \eqref{norm-bc2f} and using $ \norm{\mathcal{B}^{\sigma,\gamma}_{\C} (\varphi)}_{L^{2,\gamma}(\C)}^2 = \norm{\varphi}_{L^{2,\gamma}(\R)}^2$.

The above discussion can be reformulated as follows

 \begin{thm} \label{MainThm2}
 The bicomplex Segal--Bargmann transform ${\mathcal{B}^{\sigma,\nu}_{\BC}}$ in \eqref{BC:SBT} realizes a unitary integral transform mapping isometrically the bicomplex Hilbert space $L^{2,\sigma}_{\BC}(\R)$ onto the bicomplex Bargmann space $\Fbc$.
 \end{thm}

 \begin{rem}
A variant of the bicomplex Segal--Bargmann transform can be defined on $\Ldc$, the bicomplex Hilbert space of $\BC$-valued functions on the hyperbolic (bireal) numbers 
$$\D=\{U=x \eu + y\es; x,y\in\R \}$$ endowed with
the bicomplex inner product
\begin{align}\label{spD}
\scal{\varphi,\psi}_{L^{2,\sigma}_{\BC}(\D)} : =  c_1^{\sigma} \int_{\D}  \varphi(U) \psi(U)^{*} e^{-\sigma \norm{U}^2}d\lambda(U),
\end{align}
where $d\lambda(U)=dxdy$ is the Lebesgue measure on $\D$. In fact, we can consider
\begin{align}\label{BC:VSBT}
\widetilde{\BargBC} (\varphi )(Z) := c_1^{\frac\nu2} 
 \int_{\D}    e^{-\sigma  \left(U-\frac{\nu}{4\sigma}Z\right)^2} \varphi(U)  e^{-\sigma{U^{\dag}}^2 } d\lambda(U).
\end{align}
 For every fixed $Z$, the restriction of the involved kernel function in the right--hand side of \eqref{BC:SBT} to the diagonal of the bireal numbers reduces further to the kernel function in the bicomplex Segal--Bargmann transform $\BargBC$.
\end{rem}

\begin{rem}
It should also be noted that the kernel function ${B^{\sigma,\nu}_{\BC}}(x;Z)$ in \eqref{KernelSBT1}
 is closely connected to the generating function $G^{\sigma,\nu}(x;Z)$ of the Hermite polynomials $H_n^{\sigma}(x)$, 
 $$
B^{\sigma,\nu}_{\BC} (x;Z)  =  c_0^{\sigma} e^{-\sigma x^2} G^{\sigma,\nu}(x;Z^*) .
$$ 
 Indeed, for every given $(x;Z)\in \R \times \BC$, we have
\begin{equation}\label{GenFct}
G^{\sigma,\nu}(x;Z):=\sum_{n=0}^\infty\frac{H_n^{\sigma}(x) E_n(Z^*)}{\norm{H_n^{\sigma}}_{L^{2,\sigma}_{\BC}(\R)} \norm{E_n}_{\Hbc} }
= e^{-\frac{\nu}{4} (Z^*)^2  + \sqrt{\sigma \nu} x Z^*} .
\end{equation}
 \end{rem}

 \begin{rem}
 By means of \eqref{GenFct}, one can show that for every $Z,W\in \BC$, we have
 \begin{align*}\scal{ G^{\sigma,\nu}(\cdot;Z^*), G^{\sigma,\nu}(\cdot;W) }_{L^{2,\sigma}_{\BC}(\R)}
& =  c_0^{\sigma} \int_{\R} G^{\sigma,\nu}(x;Z^*) G^{\sigma,\nu}(x;W^*) e^{-\sigma x^2}dx
\\&  =  K^{\nu}_{\BC}(Z,W).
\end{align*}
Therefore, the function $G^{\sigma,\nu}_Z:x\mapsto{G^{\sigma,\nu}_Z(x):=G^{\sigma,\nu}(x;Z)}$, for every fixed $Z\in \BC$,
 belongs to $L^{2,\sigma}_{\BC}(\R)$. 
\end{rem}

The previous discussion shows that the transform $\BargBC$ admits an inverse mapping $\Fbc$ onto $L^{2,\sigma}_{\BC}(\R)$.
Basically, using the fact that
\begin{equation}\label{basisbasis}
\BargBC (\psi_n^{\sigma})(Z) = \phi_n(Z),
\end{equation}
where $\psi_n^{\sigma}$ (resp. $ \phi_n(Z)$) is the orthonormal basis of $L^{2,\sigma}_{\BC}(\R)$ (resp. $\Fbc$) given by \eqref{Orthonbasis} (resp. \eqref{bcbasis}), we see that the inverse transform $[\BargBC]^{-1}$ admits the expansion
\begin{align*}
[\BargBC]^{-1}(f)(x) &:= \sum_{n=0}^{\infty} c_n \norm{E_n}_{\Hbc} [\BargBC]^{-1}(\phi_n)(x)
\\&= \sum_{n=0}^{\infty} \frac{ \norm{E_n}_{\Hbc} }{ \norm{H_n^\sigma}_{L^{2,\sigma}_{\BC}(\R)}  } c_n H_n^\sigma(x)
\end{align*}
for $f := \sum\limits_{n=0}^{\infty}c_n  E_n   \in \Fbc$. 
Now, since
$\scal{E_n,E_n}_{\Hbc}= \norm{E_n}_{\Hbc}^2 \notin \mathcal{NC}$, we get $$c_n = \scal{f,E_n}_{\Hbc}(\scal{E_n,E_n}_{\Hbc})^{-1} = \frac{\scal{f,E_n}_{\Hbc}}{\norm{E_n}_{\Hbc}^2 },$$
 so that
\begin{align} \label{inverseExp}
[\BargBC]^{-1}(f)(x) &= \sum_{n=0}^{\infty}\frac{\scal{f,E_n}_{\Hbc}}{\norm{H_n^\sigma}_{L^{2,\sigma}_{\BC}(\R)} \norm{E_n}_{\Hbc} }  H_n^\sigma(x).
\end{align}
This expansion leads to the integral representation
\begin{align} \label{inverseintegral}
[\BargBC]^{-1}(f)(x) &= \scal{f, (G^{\sigma,\nu}(x;\cdot))^* }_{\Hbc}
\end{align}
that maps isometrically $\Fbc$ onto $L^{2,\sigma}_{\BC}(\R)$.
  Thus, one can assert the following result.

 \begin{thm}\label{MainThm2inverse}
 The integral representation of the unitary transform  $[\BargBC]^{-1}: \Fbc \longrightarrow L^{2,\sigma}_{\BC}(\R)$ is given by
 \begin{align}
   [\BargBC]^{-1}(f)(x) =  c_T^{\nu} \int_{\BC} e^{-\nu |Z|^2 - \frac{\nu}{4} (Z^*)^2  + \sqrt{\sigma \nu} x Z^*}  f(Z)   d\lambda(Z).\label{inverseExpIntRep}
\end{align}
 \end{thm}

In the next section, we will make use of the bicomplex Bargmann transform and its inverse to define a class of Fourier transforms generalizing the standard $i$-- and $j$--Fourier transforms. 

\section{A class of  bicomplex fractional  Fourier transforms}

The standard bicomplex Fourier transform is defined as
$$\mathcal{F} (\psi)(Z) = \frac{1}{\sqrt{2\pi}} \int_{\R} e^{-i x Z} \psi (x) dx , \quad Z=\alpha\eu+\beta \es. $$
Using the idempotent decomposition, the transform $\mathcal{F}_{\BC}$ can be seen as duplication over the two Fourier transforms with respect to complex frequencies $\alpha,\beta$ of given $\psi$ on the real line (see for example  \cite{DeBieStruppaVajiacVajiac2012,BanerjeeDattaHoque2014,BanerjeeDattaHoque2015}).

Now, for every fixed $\theta \in \{\theta\in S^1\eu+S^1\es;  \, \theta\ne\pm 1,\pm ij\}$, we define the integral transform $\mathcal{F}^\sigma_\theta $ to be
\begin{align}  \label{FrFourierT}
\mathcal{F}^\sigma_\theta \psi (y) 
= 
\frac{c_0^{\sigma}}{\sqrt{1 - \theta^2}} 
\int_{\R} \psi (x)  e^{- \frac{\sigma}{1 - \theta^2} (x-\theta y)^2  } dx
\end{align}
for given bicomplex--valued function on the real line. 
It will be called bicomplex fractional Fourier transform (BFrFT). The following result shows in particular that the kernel function of $\mathcal{F}^\sigma_\theta$ is closely connected to the integral transform $[\BargBC]^{-1} ( \Gamma_\theta \BargBC \psi )(x)$, where  $\Gamma_\theta f (Z):= f(\theta Z)$, and therefore to the bicomplex version of Mehler formula for Hermite polynomials.

\begin{thm}  \label{FourierBargmann}
The transform $\mathcal{F}^\sigma_\theta $ is well--defined on the bicomplex Hilbert space $L^{2,\sigma}_{\BC}(\R)$. Moreover, for every $\psi\in L^{2,\sigma}_{\BC}(\R)$ and $x\in \R$, we have
\begin{align}\label{tFourier}
\mathcal{F}^\sigma_\theta \psi (x) = [\BargBC]^{-1} ( \Gamma_\theta \BargBC \psi )(x) .
\end{align}
\end{thm}

\begin{proof}
Starting from the integral representations of $\BargBC $ and $[\BargBC]^{-1}$ given respectively by
$$
\BargBC(\psi)(Z) = \scal{ \psi,G^{\sigma,\nu}(\cdot;Z)}_{L^{2,\sigma}_{\BC}(\R)}
$$ and
$$
[\BargBC]^{-1}(f)(y) = \scal{f,(G^{\sigma,\nu}(y;\cdot))^*}_{\Hbc},
$$
we show that for every $\psi\in L^{2,\sigma}_{\BC}(\R)$ and $y\in \R$, we have
\begin{align*}
[\BargBC]^{-1} ( \Gamma_\theta \BargBC \psi )(y)
&=  \scal{ \scal{ \psi(\cdot\cdot) ,G^{\sigma,\nu}(\cdot\cdot;\theta \cdot )}_{L^{2,\sigma}_{\BC}(\R)}  , (G^{\sigma,\nu}(y;\cdot))^*}_{\Hbc}
\\&= \scal{ \psi(\cdot\cdot) , \scal{ G^{\sigma,\nu}(\cdot\cdot;\theta \cdot )   ,G^{\sigma,\nu}(y;\cdot)}_{\Hbc} }_{L^{2,\sigma}_{\BC}(\R)}.
\end{align*}
The involved kernel function $F^\sigma_\theta(x,y) := \scal{ G^{\sigma,\nu}(x;\theta \cdot)   ,G^{\sigma,\nu}(y;\cdot)}_{\Hbc}$ is independent of $\nu$ and can be computed explicitly using the explicit expression of $G^{\sigma,\nu}(x,Z)$ given through \eqref{GenFct}. Indeed, we obtain
\begin{align*}
F^\sigma_\theta(x,y) &=  c_T^{\nu}  \int_{\BC} e^{-\nu|Z|^2 -\frac{\nu}{4} Z^2 -\frac{\nu}{4} (\theta^*Z^*)^2 + \sqrt{\sigma \nu} y Z + \sqrt{\sigma \nu} x \theta^*Z^*}   d\lambda(Z)
\\&=  \frac{c_T^{\nu}}{4 c_1^{\frac\nu2} } \left( I_{\alpha_\theta}(x,y) e_+ + I_{\beta_\theta}(x,y) e_-\right),
\end{align*}
where $\theta =\alpha_\theta \eu + \beta_\theta \es$ and $I_{\xi}(x,y)$ stands for
$$ I_{\xi}(x,y) :=
 \int_{\C} e^{-\frac\nu 2|\zeta|^2 -\frac{\nu}{4} \zeta^2 -\frac{\nu}{4} \overline{\xi}^2 \overline{\zeta}^2 + \sqrt{\sigma \nu} y \zeta + \sqrt{\sigma \nu} x \overline{\xi} \overline{\zeta}}
  d\lambda(\zeta)  $$
  for $x,y\in\R$ and $\xi\in\C$.
Here, we recognize the integral formula
$$ \int_{\C} e^{-\gamma|\zeta|^2 + a \zeta^2 + b  \overline{\zeta}^2 + c \zeta + d \overline{\zeta}}
  d\lambda(\zeta) = \frac{\pi}{\sqrt{\gamma^2-4ab }} \exp\left(\frac{ad^2+bc^2+\gamma c d}{\gamma^2-4ab } \right)$$
valid for $|\Re(a+b)|<\gamma$, which implies in particular that $\gamma^2-4ab >0$. This can be seen as a particular case of the Gaussian integral \cite[p. 256]{Folland1989}
 \begin{align} \label{gaussIntegral}
  \int_{\R^n} e^{-a y A y + by} dy = \left(\dfrac{\pi^n}{a^n \sqrt{\det A}} \right)^{1/2} e^{\frac{1}{4 a} b A^{-1} b },
  \end{align}
 where we have the limitation $a>0$, $b\in \C^n$ and $A =(a_{mn})_{m,n} \in \C^{n\times n}$ is a symmetric $n$-matrix whose real part $\Re(A)=(\Re(a_{mn}))_{m,n}$ is positive definite.
 In our case, the limitation condition $|\Re(a+b)|<\gamma$ is equivalent to $|1+\Re(\alpha_\theta|<\gamma$ and $|1+\Re(\beta_\theta|<\gamma$ and therefore to $\theta \in \{\theta\in S^1\eu+S^1\es;  \, \theta\ne\pm 1;\pm ij\}$.
 Under this assumption, we have
\begin{align*}
F^\sigma_{\theta}(x,y)
&=  \frac{1}{c_1^{\frac\nu2} \sqrt{1 - (\theta^*)^2}}  \exp\left( \frac{- \sigma (\theta^*)^2 (x^2 + y^2) + 2 \sigma  \theta^* x y  }{1 - (\theta^*)^2}  \right).
\end{align*}
This completes our check for \eqref{tFourier} since $[\BargBC]^{-1} ( \Gamma_\theta \BargBC \psi )(x)$ reduces further to 
\begin{align*}
 \frac{c_0^{\sigma}}{\sqrt{1 - \theta^2}} e^{ - \frac{\sigma\theta^2}{1 - \theta^2}y^2}
\int_{\R} \psi (y)  e^{- \frac{\sigma}{1 - \theta^2}x^2 + 2  \frac{\sigma\theta}{1-\theta^2} xy } dx .
\end{align*}
\end{proof}

\begin{rem}
For the particular case $\theta =i$ (or also $j$, i.e. such that $\theta^2=-1$), with $\sigma=1$, the expression in \eqref{tFourier} reduces further to
the standard Fourier transform 
 \begin{align}\label{tFourier}
 \psi \longmapsto \left( y \longmapsto \frac{1}{\sqrt{2\pi}}  e^{\frac{y^2}{2}}
\int_{\R}   e^{ i x y  } \psi(x) e^{\frac{-y^2}{2}}  dx \right)
\end{align}
for every $\psi\in L^{2}_{\BC}(\R;e^{-x^2}dx)$ and $x\in \R$.
\end{rem}

\begin{rem}
 Using the generating nature of $G^{\sigma,\nu}$ involving the functions $\psi_n $ and $ \phi_n$ given by \eqref{Orthonbasis} and \eqref{bcbasis}, respectively, combined with the orthonormality of $\psi_n$ in $\Hbc$, we can rewrite the kernel $F^\sigma_{\theta}(x,y)$ as
\begin{align*}
F^\sigma_{\theta}(x,y) &= \sum_{n=0}^\infty \frac{\theta^n H_n^\sigma(x) H_n^\sigma(y)}{2^n\sigma^n n!}.
\end{align*}
The right-hand side is the bicomplex version of Mehler formula for the rescaled real Hermite polynomials
\begin{align}\label{MehlerkernelHnsigma}
\sum_{m=0}^\infty \frac{\theta^n H_n^\sigma(x) H_n^\sigma(y)}{2^n\sigma^n n!}
= \frac{1}{\sqrt{1 - \theta^2}}  \exp\left( \frac{- \sigma \theta^2 (x^2 + y^2) + 2 \sigma  \theta x y  }{1 - \theta^2}  \right)
\end{align}
valid for every fixed $\theta \in \BC$ such that $|\theta|<1$. It can be obtained easily form the classical one \cite{Mehler1866,Rainville71,Andrews}.
However, our approach shows that the expected formula is also valid for $\theta \in \{\theta\in S^1\eu+S^1\es;  \, \theta\ne\pm 1;\pm ij\}$.
\end{rem}

The uniqueness theorem as well as the Plancherel theorem and the inversion formula for the  bicomplex fractional  Fourier transform $\mathcal{F}^\sigma_\theta$ immediately follow in virtue of Theorem \ref{FourierBargmann}. Thus, we assert

\begin{cor}
If for given $\psi_1,\psi_2\in  L^{2,\sigma}_{\BC}(\R)$ we have $\mathcal{F}^\sigma_\theta \psi_1= \mathcal{F}^\sigma_\theta \psi_2$, then $\psi_1=\psi_2$.
\end{cor}

\begin{cor}
For every $\psi \in  L^{2,\sigma}_{\BC}(\R)$, we have $\norm{\mathcal{F}^\sigma_\theta \psi}_{L^{2,\sigma}_{\BC}(\R)}= \norm{\psi}_{L^{2,\sigma}_{\BC}(\R)}$ and $\mathcal{F}^\sigma_\theta \psi_n=\theta^n \psi_n$.
\end{cor}

\begin{cor}
The inversion formula for $\mathcal{F}^\sigma_\theta$ is the FrFT with the parameter $\theta^*$. More explicitly
\begin{align}  \label{FrFourierTinverse}
[\mathcal{F}^\sigma_\theta]^{-1}( \psi)(x) = 
 \frac{c_0^{\sigma}}{\sqrt{1 - (\theta^*)^2}} 
\int_{\R} \psi (y)  e^{- \frac{\sigma}{1 - (\theta^*)^2}(y - \theta^*x)^2 } dy 
\end{align}
for every $\psi \in L^{2,\sigma}_{\BC}(\R)$.
\end{cor}

  \begin{rem}
The inversion formula in \eqref{FrFourierTinverse} follows since
$$[\mathcal{F}^\sigma_\theta]^{-1} = [\BargBC]^{-1} \Gamma_{\theta^*} \BargBC $$
for every fixed bicomplex $|\theta|=1$ with $\theta\ne \pm 1, \pm ij$. It can also be seen as an immediate consequence of the semi-group property $\mathcal{F}^\sigma_\theta \circ \mathcal{F}^\sigma_{\varrho}= \mathcal{F}^\sigma_{\theta\varrho}$ which readily follows from  Theorem \ref{FourierBargmann} and the fact that $\Gamma_{\theta_1}\circ \Gamma_{\theta_2} = \Gamma_{\theta_1\theta_2}$.
\end{rem}

\section{Concluding remarks}
  \begin{rem}
  In our investigation, the condition $|\theta|=1$, with $\theta \ne \pm 1,\pm ij$,  is needed to deal with the bicomplex analog of the classical fractional Fourier transform. However, the kernel function $T^{\sigma}_\theta(x,y)$ in \eqref{MehlerkernelHnsigma},
obtained by means of the bicomplex version of Mehler formula,  can be used to define the integral transform
  $$\widetilde{\mathcal{F}^\sigma_\theta}(\varphi)(y) = \int_\R \varphi(x) T^{\sigma}_\theta(x,y) dx $$
for $y\in \R$, $\theta\in \BC$ such that $|\theta|<1$ and $\varphi \in L^{2,\sigma}_{\BC}(\R)$.
In this case, the connection to the bicomplex Segal--Bargmann transform can be shown to be given by
$$\widetilde{\mathcal{F}^\sigma_\theta}  =   [\mathcal{B}^{\sigma,\nu \gamma_\theta}_{\BC}]^{-1} \Gamma_\theta  \mathcal{B}^{\sigma,\nu}_{\BC}$$
for some positive real $\gamma_\theta$ depending only on $\theta$.
\end{rem}

 \begin{rem}
 The authors of \cite{DeBieStruppaVajiacVajiac2012} define the Fourier transform for bicomplex holomorphic functions as the Cauchy-Kowalewski extension of the classical one on the real line (similarly as it has been done in the setting of Clifford analysis in $\R^m$ in \cite{LiMcIntoshQian1994}), namely for given  bicomplex holomorphic function $F(Z)$, the bicomplex Fourier transform is given by
$$\mathcal{F}_{\BC}(F)(Z) = \frac{1}{\sqrt{2\pi}} \int_{\R} e^{-i x Z} F|_{\R}(x) dx .$$
By proceeding in a similar way, we define the Cauchy-Kowalewski extension of the BFrFT for bicomplex holomorphic functions by considering
\begin{align*}  \label{Cauchy-KowalewskiFrFourierT}
\mathcal{F}^{\sigma,\theta}_{\BC} F (Z) =  
 \frac{c_0^{\sigma}}{\sqrt{1 - \theta^2}} e^{ - \frac{\sigma\theta^2}{1 - \theta^2}Z^2}
\int_{\R} F|_{\R}(x)  e^{- \frac{\sigma}{1 - \theta^2}x^2 + 2  \frac{\sigma\theta}{1-\theta^2} x Z} dx .
\end{align*}
The kernel function of the integral transform $\mathcal{F}^{\sigma,\theta}_{\BC}$ is closely connected to the bicomplex version of the bilinear Mehler formula involving the product of the rescaled real Hermite polynomials $ H_n^{\sigma}(x)$ in \eqref{HermitPol}
  and the bicomplex holomorphic Hermite polynomials
  $H_n^\sigma(Z)$ obtained from the first one by replacing $y$ by the bicomplex $Z$, being indeed
\begin{equation}\label{MehlerkernelHnsigma}
\sum_{n=0}^\infty \frac{\theta^n H_n^\sigma(Z) H_n^\sigma(y)}{2^n\sigma^n n!}
= \frac{1}{\sqrt{1 - \theta^2}}  \exp\left( \frac{- \sigma \theta^2 (Z^2 + y^2) + 2 \sigma  \theta y Z }{1 - \theta^2}  \right).
\end{equation}
Basic properties of this transform will be discussed in a forthcoming investigation.
\end{rem}

 \begin{rem}
  By applying the approach, we have adopted in introducing the bicomplex fractional Fourier transform $\mathcal{F}^\sigma_\theta$ to the variant of the bicomplex Segal--Bargmann transform $\mathcal{\BargBC}$ in \eqref{BC:VSBT}, we are able to define and study a bicomplex fractional Fourier transform on the bireal set $\D$. We hope to return to this point in a near future.
\end{rem}

\noindent{\bf Acknowledgments:}
A part of this work is done during the visit of the first author to Dipartimento di Matematica Politecnico di Milano (May--June 2017). 
The assistance of the members of "Ahmed Intissar's seminar on Analysis, Partial Differential Equations and Spectral Geometry" is gratefully acknowledged.

\end{document}